\documentclass[12pt]{amsart}
\numberwithin{equation}{section}
\usepackage{amssymb}
\usepackage{amsmath}
\usepackage{latexsym}
\newcommand{\eps}{\varepsilon}

\def\C{\mathbb C}

\def\CC{\widehat{\mathbb C}}
\def\N{\mathbb N}

\def\dim{\operatorname{dim}}

\def\dist{\operatorname{dist}}

\def\diam{\operatorname{diam}}
\def\dens{\operatorname{dens}}
\def\area{\operatorname{area}}

\def\Sing{{\rm sing}}
\def\HD{\text{{\rm HD}}}
\def\es{\emptyset}
\def\sms{\setminus}
\def\sbt{\subset}
\def\spt{\supset}

\def\ov{\overline}
\def\a{\alpha}
\def\b{\beta}
\def\d{\delta}

\def\l{\lambda}

\def\sg{\sigma}

\newtheorem{la}{Lemma}[section]
\newtheorem{thm}{Theorem}[section]

\theoremstyle{definition}

\theoremstyle{remark}

\title[On the Hausdorff dimension of the escaping set]
{On the Hausdorff dimension of the escaping set of certain
meromorphic functions}

\thanks{The authors were supported by the EU Research Training
Network CODY. The first author was also supported by the G.I.F., the
German--Israeli Foundation for Scientific Research and Development,
Grant G-809-234.6/2003 and the ESF Research Networking Programme
HCAA. The second author was also supported by Polish MNiSW Grant N
N201 0222 33 and PW Grant 504G 1120 0011 000.}
\author{Walter Bergweiler}
\address{Mathematisches Seminar,
Christian--Albrechts--Universit\"at zu Kiel,
Lude\-wig--Meyn--Str.~4, D--24098 Kiel, Germany}
\email{bergweiler@math.uni-kiel.de}
\author{Janina  Kotus}
\address{Faculty of Mathematics and Information Science,
Warsaw University of Technology, Pl.\ Politechniki 1, 00-661
Warszawa, Poland} \email{J.Kotus@impan.pw.edu.pl}
\date{\today}

\begin{document}

\begin{abstract}
Let $f$ be  a transcendental  meromorphic function
of finite order  $\rho$ for which the set of finite singularities of $f^{-1}$ is bounded.
Suppose that $\infty$ is not an
asymptotic value and that there exists $M \in \mathbb N$ such that
the  multiplicity of all  poles, except possibly finitely many, is
at most $M$. For $R>0$ let $I_R(f)$ be the set of all $z\in\C$ for which
$\liminf_{n\to\infty}|f^n(z)|\geq R$ as $n\to\infty$. Here $f^n$ denotes
the $n$-th iterate of $f$.
Let $I(f)$ be the set of all $z\in\C$ such that  $|f^n(z)|\to\infty$ as
$n\to\infty$; that is, $I(f)=\bigcap_{R>0} I_R(f)$.
Denote the Hausdorff dimension of a set $A\subset\C$ by $\HD(A)$.
It is shown that
$\lim_{R \to \infty} \HD(I_R(f))\leq 2 M \rho/(2+ M\rho)$.
In particular, $\HD(I(f))\leq 2 M \rho/(2+ M\rho)$.
These estimates are best possible: for given $\rho$ and $M$ we construct
a function $f$ such that $\HD(I(f))= 2 M \rho/(2+ M\rho)$ and
$\HD(I_R(f))> 2 M \rho/(2+ M\rho)$ for all $R>0$.

If $f$ is as above but of infinite order, then the area of $I_R(f)$ is zero.
This result does not hold without a restriction on the multiplicity of the poles
\end{abstract}

\maketitle

\section{Introduction and main results}

The  Fatou set $F(f)$ of a (non-linear)
function $f$ meromorphic in the plane is
defined as the set of all points $z\in \mathbb C$ such that   the
iterates  $f^k$   of $f$ are  defined and form  a normal family  in
some neighbourhood of $z$. Furthermore, $J(f)=\CC\sms F(f)$
where $\CC=\C\cup\{\infty\}$ is
called the Julia set  of $f$ and
\[
 I(f) = \{z\in\C: f^n(z) \to \infty \mbox{ as } n \to \infty \}
\]
is called  the escaping set of   $f$. In addition to these sets, we
shall  also  consider for $R>0$ the set
\[
I_R(f)=\{z \in {\mathbb C}:
  \liminf_{n\to \infty}|f^n(z)|\geq R\}.
\]
Note that
\[
  I(f)=\bigcap_{R>0}I_R(f).
\]
It was shown by Eremenko~\cite{Ere89} for entire $f$ and
by Dom\'inguez~\cite{Dom98} for transcendental mero\-morphic $f$ that
$I(f)\neq\emptyset$ and that $J(f)=\partial I(f)$.
For an introduction to the iteration theory of transcendental
meromorphic functions we refer to~\cite{Ber93}. Results on
the Hausdorff dimension of Julia sets and related sets
are surveyed in~\cite{KU2,Stall}.

The  set of singularities  of the inverse function $f^{-1}$ of $f$
coincides with the set of critical values and  asymptotic values of
$f$. We denote  the set of  finite  singularities  of $f^{-1}$ by
$\Sing(f^{-1})$. The Eremenko-Lyubich class $\mathcal B$ consists of
all meromorphic functions  for which  $\Sing(f^{-1})$  is bounded.
Eremenko and Lyubich~\cite[Theorem~1]{EL} proved that if
$f\in \mathcal B$ is entire,
then $I(f)\sbt J(f)$. This result  was  extended  to  meromorphic
functions in $\mathcal B$ by Rippon and Stallard~\cite{RS}. Actually
the proof yields that $I_R(f) \sbt J(f)$ if  $f \in \mathcal B$ and
$R$ is sufficiently large.

For $A\sbt \mathbb C$  we denote by $\HD(A)$  the Hausdorff
dimension of $A$ and by $\area(A)$ the two-dimensional Lebesgue
measure of $A$. McMullen~\cite{McM}  proved hat $\HD(J(\l e^z))=2$
for $\l \in {\mathbb C}\sms \{0\}$ and  that $\area( J(\sin (\a z+
\b z)))> 0 $ for $\a, \b \in {\mathbb C}$, $\a \neq 0$. His proof shows
that the conclusion holds  with $J(\cdot)$ replaced by $I(\cdot)$. Note that
the functions considered  by McMullen are in the class $\mathcal B$
so that  the escaping set is  contained in the Julia  set.

The order  $\rho(f)$  of a  meromorphic  function $f$  is defined by
\[
\rho(f)=\limsup_{r \to \infty}\frac{\log T(r,f)}{\log r}
\]
where $T(r,f)$   denotes  the Nevanlinna characteristic  of $f$;
see~\cite{GO,Hay64,Nev} for the notations of Nevanlinna theory. If
$f$  is entire, then  we may replace  $T(r,f)$   by $\log M(r,f)$
here, where $M(r,f)= \max_{z=r}|f(z)|$.  Thus for entire $f$ we
have~\cite[p.~18]{Hay64}
\[
\rho(f)=\limsup_{r \to \infty}\frac{\log\log M(r,f)}{\log r}.
\]
It is easy  to see  that $\rho(\l e^z)=\rho(\sin (\a z+\b z))=1$  for
$ \l, \a, \b \in \mathbb C$, $ \l, \a \neq 0$.

McMullen's result
that $\HD(J(\l e^z))=2$ was  substantially generalized by
Bara\'nski~\cite{Bar} and Schubert~\cite{Sch07}  who proved  that if
$f\in \mathcal B$ is entire  and $\rho(f) < \infty$, then
$\HD(J(f))=2$. In fact, they show  that $\HD(I_R(f))=2$  for all
$R>0$ under these hypotheses.
Their proofs, which make use of the  logarithmic  change of variable
introduced by Eremenko and Lyubich, show that  the conclusion  holds
more generally for  meromorphic  functions in   $\mathcal  B$ which
have  finite order and  for which  $\infty$ is an asymptotic value.
In fact, such  functions   have a logarithmic singularity  over
$\infty$ and  their dynamics   are in  many ways similar  to those
of entire  functions; see, e.g., \cite{BKZ} or~\cite{BRS}.

The  purpose of this paper  is to show that the situation is very
different for  meromorphic functions of class $\mathcal B$ for which
$\infty$ is not an  asymptotic value.

\

\begin{thm}\label{thm1}
Let $f \in \mathcal B$ be  a transcendental  meromorphic function
with $\rho=\rho(f)< \infty$. Suppose that $\infty$ is not an
asymptotic value  and that  there exists $M \in \mathbb N$ such that
the  multiplicity of all  poles, except possibly finitely many, is
at most $M$. Then
\begin{equation}\label{1a}
\HD(I(f))\leq \frac{2 M \rho}{ 2+ M \rho}
\end{equation}
and
\begin{equation}\label{1b}
\lim_{R \to \infty} \HD(I_R(f))\leq \frac{2 M \rho}{ 2+ M\rho}.
\end{equation}
\end{thm}

\

Note that $I_{S}(f) \sbt I_R(f)$ if $S> R$. Hence
$\HD(I_R(f))$ is a non-increasing  function of $R$ and thus the  limit in
(\ref{1b}) exists. Clearly~(\ref{1a}) follows from~(\ref{1b}) so
that it suffices to prove~(\ref{1b}).

We note that elliptic functions are in ${\mathcal B}$ and have order~$2$.
It was shown in~\cite[Theorem~1.2]{KU} that if
$M$ denotes the maximal multiplicity of the poles of an elliptic
function~$f$, then $\HD(I(f))\leq 2M/(1+M)$. Inequality~\eqref{1a}
generalizes this result.

On the other hand, it was shown in~\cite[Example~3]{Kotus95} that
if $f$ is an elliptic function such that the closure of
the postcritical set is disjoint from the set of poles,
then $\HD(J(f))\geq 2M/(1+M)$. The argument shows that
$\HD(I(f))\geq 2M/(1+M)$.
Thus~\eqref{1a} is best possible if $\rho=2$.
The  following result  shows  that Theorem~\ref{thm1}  is best
possible for all values of~$\rho$.

\

\begin{thm}\label{thm2}
Let $0< \rho < \infty$  and  $M \in \mathbb N$. Then
there exists a  meromorphic function $f\in{\mathcal B}$
of order $\rho$ for which
all poles have  multiplicity $M$ and for which $\infty$ is not  an
asymptotic value such that
\begin{equation}\label{1c0}
\HD(I(f))= \frac{2 M \rho}{ 2+ M \rho}
\end{equation}
and
\begin{equation}\label{1c}
\HD(I_R(f))>  \frac{2 M \rho}{ 2+ M \rho}
\end{equation}
for all $R>0$.
\end{thm}

\

For functions of infinite  order we cannot expect the Hausdorff
dimension of $J(f)$ or  $I_R(f)$ to be  less  than 2. However, we
have the  following  result.

\

\begin{thm}\label{thm3}
Let $f \in \mathcal B$ be  a transcendental  meromorphic for which
 $ \infty$ is not an asymp\-to\-tic value.  Suppose
that  there exists $M \in \mathbb N$ such that   all poles of $f$
have multiplicity  at most $M$. Then
\[
\area\left( I_R(f)\right)=0
\]
for sufficiently large $R$. In particular,
\[
\area\left( I(f)\right)=0.
\]
\end{thm}

\

The proof of Theorem~\ref{thm3} uses well-known techniques, see~\cite{Stallard90}
for a similar argument. In fact, as kindly pointed out to us by Lasse Rempe,
Theorem~\ref{thm3} is implicitly contained in~\cite[Theorem~7.2]{RempeVanStrien}.
However, we shall include the short proof of Theorem~\ref{thm3} for completeness.

Finally  we show that the hypothesis  on the  multiplicity of the
poles is essential.

\

\begin{thm}\label{thm4}
There exists  a transcendental meromorphic $ f \in \mathcal B$  for
which  $ \infty$ is not an asymptotic value  and for  which
\[
\area( I(f))>0.
\]
\end{thm}

\

\section{Notations and preliminary Lemmas}
The diameter of a set $K\subset \C$ is denoted by $\diam(K)$. Later we will
also use the area and diameter with respect to the spherical metric~$\chi$.
We will denote them by $\area_{\chi}(K)$ and $\diam_{\chi}(K)$, respectively.

For $ a \in \C$ and $r, R>0$ we  use  the  notation $D(a,r)=\{z\in
\C:  |z-a|<r\}$  and
\[B(R)=\{z \in \C:  |z|>R\}\cup \{\infty\}.
\]
The following  lemma is known as Koebe's distortion theorem and
Koebe's $\frac{1}{4}$-theorem.

\

\begin{la}\label{2.1}
Let $g:D(a,r)\to \C$ be univalent, $0< \lambda <1$ and $z\in D(a,\lambda
r)$. Then
\begin{equation}\label{2a}
\frac{\lambda}{(1+\lambda)^2}|g'(a)|r\leq |g(z)-g(a)|\leq
\frac{\lambda}{(1-\lambda)^2}|g'(a)|r,
\end{equation}
\begin{equation}\label{2b}
\frac{1-\lambda}{(1+\lambda)^3}|g'(a)|\leq |g'(z)|\leq
\frac{1+\lambda}{(1-\lambda)^3}|g'(a)|r
\end{equation}
and
\begin{equation}\label{2c}
g(D(a,r))\spt D\left(g(a),\tfrac{1}{4}|g'(a)|r\right).
\end{equation}
\end{la}

\

Koebe's  theorem is usually only stated  for  the special  case that $a=0$,
$r=1$, $g(0)=0$  and $g'(0)=1$, but the above  version follows
immediately from this special case.

\

 The following  result is due to Rippon and
 Stallard~\cite[Lemma 2.1]{RS}.

\

\begin{la}\label{2.2}
Let $f\in \mathcal  B$ be transcendental. If $ R>0$ such that
$\Sing(f^{-1}) \sbt D(0,R)$, then all components  of $f^{-1}(B(R))$
are simply-connected. Moreover, if $\infty$ is not an asymptotic
value of $f$, then all  components of $f^{-1}(B(R))$ are  bounded
and contain exactly one pole  of $f$.
\end{la}

\

The following result is known as Iversen's theorem~(\cite[p.~171]{GO} or~\cite[p.~292]{Nev}).

\

\begin{la}\label{2.3}
Let $f$ be a transcendental
meromorphic function for which $\infty$ is not an
asymptotic value. Then $f$  has infinitely  many poles.
\end{la}

\

Let $(a_j)$ be a sequence of non-zero complex numbers such that
$\lim_{j\to \infty} |a_j|=\infty$. Then
\[
\sg=\sg((a_j))=\inf\left\{ t>0: \sum_{j=1}^\infty |a_j|^{-t}< \infty
\right\}
\]
is called  the exponent  of convergence of the sequence $(a_j)$.
Here we  use the convention that
$\inf\es=\infty$, meaning  that
$\sg=\infty$ if  $\sum_{j=1}^\infty |a_j|^{-t}= \infty$  for all
$t>0$.

\

The  following  lemma  is standard~\cite[p.~26]{Hay64}.

\

\begin{la}\label{2.4}
Let $f$ be a transcendental  meromorphic function and let $ \sg$ be
the exponent of convergence of the  non-zero poles of $f$. Then $\sg
\leq \rho(f)$.
\end{la}

\

We  mention  that a result of Teichm\"uller~\cite{Tei}
says that if $f \in
\mathcal B$ is  transcendental, if $\infty$  is not  an asymptotic
value of $f$ and if there exists $M\in \N$ such  that all  poles of
$f$ have multiplicity at most $M$, then $m(r,f)=O(1)$ as $r\to
\infty$. This easily implies that the exponent of convergence of the
non-zero poles of $f$ is actually  equal to $\rho(f)$ in this case.

\

\section{Proof of Theorem~\ref{thm1}}\label{proofthm1}

By Lemma~\ref{2.3}, $f$  has  infinitely many poles.
Let $(a_j)$ be  the sequence  of poles $f$,
ordered  such that $|a_j|\leq  |a_{j+1}|$ for all $j$,
and let $m_j$  be the   multiplicity of
$a_j$. Then
\[
f(z)\sim\left(\frac{b_j}{z-a_j}\right)^{m_j} \quad \mbox{as}\quad
z\to a_j
\]
for some $b_j \in \C \sms\{0\}$. We  may  assume that  $|a_j| \geq
1$ for all $j\in\N$. Let $R_0>1$ such that $\Sing(f^{-1})\sbt D(0, R_0)$
and $|f(0)|<R_0$.

Lemma~\ref{2.2}  says that if $R\geq R_0$, then  all  components of
$f^{-1}(B(R))$  are  bounded  and simply-connected and each
component contains
exactly one pole.  We denote the component containing  $a_j$ by
$U_j$ and  choose a conformal map
$\phi_j: U_j \to D(0, R^{-1/m_j})$
satisfying $\phi_j(a_j)=0$. Then   $|f(z) \phi_j(z)^{m_j}|\to 1$ as $z
$ approaches  the  boundary  of $U_j$. Since  $|f(z)
\phi_j(z)^{m_j}|$ remains  bounded near $a_j$ and is non-zero in
$U_j$, we  deduce from  the maximum principle that $|f(z)
\phi_j(z)^{m_j}|=1$  for all  $z\in U_j\sms \{a_j\}$ and that
$|\phi_j'(a_j)|=1/|b_j|$.
We may  actually  normalize $\phi_j$
such that
$\phi_j'(a_j)=1/b_j$.
Denote the inverse function of
$\phi_j$ by $\psi_j$. Since $\psi_j(0)=a_j$ and $\psi_j'(0)=b_j$ we deduce from~(\ref{2c})
that
\begin{equation}\label{3a}
U_j=\psi_j(D(0, R^{-1/m_j})) \supset D\left(a_j, \frac{1}{4}|b_j|
R^{-1/m_j}\right)\supset D\left(a_j, \frac{1}{4R}|b_j| \right).
\end{equation}
Since $|f(0)|<R$ we have $0\notin U_j$. Thus~\eqref{3a} implies in  particular that
\[
\frac{1}{4R}|b_j| \leq |a_j|
\]
for all $R \ge R_0$ and hence that
\begin{equation}\label{3b}
|b_j|\leq 4 R_0 |a_j|.
\end{equation}
We  note that $\psi_j$  actually extends to a  map univalent in
$D(0,R_0^{-1/m_j})$.
Applying~(\ref{2a}) with
\[
\lambda=\left(\frac{R}{R_0}\right)^{-1/m_j}=\left(\frac{R_0}{R}\right)^{1/m_j}
\]
we find that
\[
U_j\sbt D\left(a_j,
\frac{\lambda}{(1-\lambda)^2}|b_j|R^{-1/m_j}\right).
\]
Choosing $R \geq  2^M R_0$ we  have $\lambda \leq \frac{1}{2} $ and
hence
\begin{equation}\label{3c}
U_j\sbt D\left(a_j,2|b_j| R^{-1/M}\right),
\end{equation}
provided  $j$  is so large  that $m_j\leq M$. Combining (\ref{3a})
and (\ref{3c}) we  thus have
\begin{equation*}\label{3d}
D\left(a_j,\frac{1}{4R}|b_j|\right)  \sbt U_j \sbt D\left(a_j,2
R^{-1/M}|b_j| \right)
\end{equation*}
for large $j$.
Combining (\ref{3b}) and (\ref{3c}) we  see that
\[
U_j \sbt D\left(a_j,8 R_0|a_j|R^{-1/M} \right).
\]
Choosing $R \geq \left(16 R_0\right)^M$  we thus have
\begin{equation}\label{3e}
U_j \sbt D\left(a_j,\frac{1}{2} |a_j| \right)
 \sbt D\left(0,\frac{3}{2} |a_j| \right).
\end{equation}
Next  we note  that the $U_j$  are pairwise  disjoint. Combining
this with (\ref{3a}) and  (\ref{3e}) we see  that if $n(r)$  denotes
the
 number of $a_j$  contained in the closed disc $\ov{D(0,r)}$, then
\[
\begin{aligned}
\frac{\pi}{16 R^2}\sum_{j=1}^{n(r)}|b_j|^2&=\area\left(
\bigcup_{j=1}^{n(r)}D\left(a_j, \frac{1}{4R}|b_j|\right)\right)\\
&\leq   \area\left( \bigcup_{j=1}^{n(r)} U_j
\right)\\
&\leq   \area \left(
D\left(0,\frac{3}{2}r\right)\right)\\
& = \frac{9\pi}{4} r^2.
\end{aligned}
\]
Hence
\begin{equation}\label{3f}
\sum_{j=1}^{n(r)}|b_j|^2\leq 36R^2 r^2 .
\end{equation}

\

We  shall use~(\ref{3f}) to prove the following result.

\

\begin{la}\label{3.1}
If
\[
t>  \frac{2M \rho}{2+ M \rho},
\]
then
\[
\sum_{j=1}^\infty\left(\frac{|b_j|}{|a_j|^{1+1/M}}\right)^t <
\infty.
\]
\end{la}

\begin{proof}
We put
\[
s= \frac{\rho}{2}\left( \frac{t}{2}-1\right) + 1 +
\frac{t}{2M}.
\]
Then
\[
s>  \frac{\rho}{2}\left( \frac{M \rho}{2+ M \rho}-1\right) + 1+
\frac{\rho}{2+M\rho}=1.
\]
For $l \geq 0$ we  put
\[
P(l)=\left\{j\in\N:  n\left(2^l\right) \leq j < n\left(2^{l+1}\right)\right\}
= \left\{j\in\N:  2^l\leq |a_j|< 2^{l+1} \right\}
\]
and
\[
S_l = \sum_{j\in P(l)}\left(\frac{|b_j|}{|a_j|^{1+1/M}}\right)^t
= \sum_{j\in P(l)}\left(\frac{|b_j|}{|a_j|^{s}}\right)^t
\left(\frac{1}{|a_j|}\right)^{t(1-s+1/M)}.
\]
We now apply H\"older's inequality, with  
$p=2/t$ and $q=2/(2-t)$. 
Putting
\[
\a=t\left(1-s+\frac{1}{M}\right) \frac{2}{2-t}=t \frac{2M\rho
+2}{2M}
> \rho
\]
we   obtain
\[
S_l\leq \left( \sum_{j\in P(l)}
\frac{|b_j|^2}{|a_j|^{2s}}\right)^{t/2} \left( \sum_{j\in P(l)}
\frac{1}{|a_j|^{\a}}\right)^{(2-t)/2}.
\]
Since $\a > \rho$  the series $\sum_{j=1}^\infty|a_j|^{-\a}$
converges by Lemma~\ref{2.4}. Thus
\[
\left(\sum_{j \in P(l)}
\frac{1}{|a_j|^{\a}}\right)^{(2-t)/2}\leq A:=
\left(\sum_{j=1}^ \infty
\frac{1}{|a_j|^{\a}}\right)^{(2-t)/2} <\infty.
\]
We now see, using (\ref{3f}), that
\[
\begin{aligned}
S_l
&
\leq  A \left(\sum_{j\in P(l)}
\frac{|b_j|^2}{|a_j|^{2s}}\right)^{t/2}
\\ &
\leq A \left( \frac{1}{(2^l)^{2s}}\sum_{j\in P(l)}
|b_j|^2\right)^{t/2}
\\ &
\leq  \frac{A}{2^{lst}}
\left(36R^2 2^{2(l+1)}\right)^{t/2}
\\ &
= A(12R)^t \left( 2^{t(1-s)}\right)^l.
\end{aligned}
\]
Since $t(1-s)< 0$, the series $\sum_{l=0}^\infty S_l$  converges.
\end{proof}


Continuing with the  proof of Theorem~\ref{thm1} we  note  that in
each simply-connected  domain $D \sbt B(R)\sms \{\infty\}$ we can
define all  branches of the  inverse  function of~$f$. Let $g_j$ be a
branch of $f^{-1}$ that maps $D$ to $U_j$. Thus
\begin{equation}\label{3g}
g_j(z)=\psi_j\left(\frac{1}{z^{1/m_j}}\right)
\end{equation}
for some  branch of the $m_j$-th root. We  obtain
\[
g_j'(z)=-\psi_j'\left(\frac{1}{z^{1/m_j}}\right)\frac{1}{m_j
z^{1+1/m_j}}.
\]
Since  we assumed  that $R \geq 2^M R_0$  we  deduce from (\ref{2b})
with $\lambda=\frac{1}{2}$ that
\begin{equation}\label{3g1}
|g_j'(z)|\leq \frac{12 |\psi_j'(0)|}{|z|^{1+1/M}} = \frac{12
|b_j|}{|z|^{1+1/M}},
\end{equation}
for $z\in D \sbt B(R) \sms \{\infty\}$, provided   $j$  is so large
that $m_j \leq M$.
From~(\ref{3c}) we deduce that
\[
\diam(U_k)\leq \frac{4}{R^{1/M}}|b_k|.
\]
Moreover, if $U_j\sbt B(R)$,
then
\[
\begin{aligned}
\diam g_j(U_k)& \leq \sup_{z \in U_k}|g_j'(z)| \diam U_k\\
&  \leq \frac{12 |b_j|}{(\frac{1}{2}|a_k|)^{1+1/M}}
\frac{4}{R^{1/M}}|b_k|\\
&= 2^{1/M} 24 \frac{4}{R^{1/M}}|b_j|\frac{|b_k|}{|a_k|^{1+1/M}}.
\end{aligned}
\]
Induction shows that if
$U_{j_1},U_{j_2},\ldots,U_{j_l}\subset B(R)$, then
\begin{equation}\label{3h}
\begin{aligned}
&\quad \
\diam \left(\left(g_{j_1}\circ g_{j_2}\circ  \ldots \circ
g_{j_{l-1}}\right) (U_{j_l})\right)\\
&\leq
(2^{1/M}
24)^{l-1}\frac{4}{R^{1/M}}|b_{j_1}|\frac{|b_{j_2}|}{|a_{j_2}|^{1+1/M}}\ldots
\frac{|b_{j_l}|}{|a_{j_l}|^{1+1/M}}.
\end{aligned}
\end{equation}
In order to obtain an estimate for the spherical diameter, we estimate
the spherical distance
$\chi(z_1,z_2)$
of  two points
$z_1, z_2 \in D(a_j, \frac{1}{2}|a_j|)$. We have
\[
\chi(z_1, z_2)
=\frac{2|z_1-z_2|}{\sqrt{1+|z_1|^2}\sqrt{1+|z_2|^2}}
\leq \frac{2|z_1-z_2|}{1+\frac{1}{4}|a_j|^2}
\leq \frac{8|z_1-z_2|}{1+|a_j|^2}
\leq \frac{8|z_1-z_2|}{|a_j|^{1+1/M}}.
\]
Thus
\[
\diam_{\chi}(K)\leq \frac{8}{|a_j|^{1+1/M}} \diam(K)
 \]
for $K \sbt U_j$ and hence (\ref{3h}) yields
\begin{equation}\label{3i}
\diam_{\chi} \left(\left(g_{j_1}\circ g_{j_2}\circ  \ldots \circ
g_{j_{l-1}}\right) (U_{j_l})\right)
\leq
(2^{1/M} 24)^{l-1} \frac{32}{R^{1/M}} \prod_{k=1}^l
\frac{|b_{j_k}|}{|a_{j_k}|^{1+1/M}}.
\end{equation}
Now there are $m_{j_k}$  branches of the inverse function of $f$
mapping $U_{j_{k+1}}$ into  $U_{j_k}$, for $k=1,2, \ldots, l-1$. Overall
we  see  that  there  are
\[
\prod_{k=1}^{l-1} m_{j_k}\leq M^{l-1}
\]
sets of diameter bounded  as  in (\ref{3i})  which  cover all those
components $V$ of $f^{-l}(B(R))$ for which  $f^{k}(V)\sbt
U_{j_{k+1}}\subset B(R)$ for $k=0,1, \ldots, l-1$.
We denote by $E_l$ the collection of all components $V$ of
$f^{-l}(B(R))$ for which  $f^{k}(V)\subset B(R)$ for $k=0,1,\ldots,l-1$.

Next we note that
(\ref{3e}) implies that if $U_j\cap B(3R)\neq \es,$ then $|a_j| >
2R$ and $U_j\subset B(R)$.
We conclude that $E_l$ is a cover of the set
\[
\{z\in B(3R): f^k(z)\in B(3R)\ \mbox{for}\  1\leq  k \leq
l-1\}.
\]
Moreover, if
$t>2 M\rho/(2+ M\rho)$, then
\[
\begin{aligned}
\sum_{V\in E_l}\left(\diam_{\chi}(V)\right)^t
& \leq M^{l-1} \left((2^{1/M}
24)^{l-1}\frac{32}{R^{1/M}}\right)^t \sum_{j_1=n(R)}^\infty\ldots
\sum_{j_l=n(R)}^\infty \prod_{k=1}^l\left(
\frac{|b_{j_k}|}{|a_{j_k}|^{1+1/M}}\right)^t\\
& = \frac{1}{M}\left(\frac{32}{(2R)^{1/M}24}\right)^t
\left( M (2^{1/M} 24)^{t}
\sum_{j=n(R)}^\infty
\left(\frac{|b_j|}{|a_j|^{1+1/M}}\right)^t \right)^l.
\end{aligned}
\]
Lemma~\ref{3.1} implies that
\[
M(2^{1/M}24)^t \sum_{j=n(R)}^\infty
\left(\frac{|b_j|}{|a_j|^{1+1/M}}\right)^t <
1
\]
for large $R$.
For such $R$  we find  that
\[
\lim_{l \to \infty} \sum_{V\in E_l}\left(\diam_{\chi}(V)\right)^t=0
\]
and thus
\[
\HD\left(\left\{ z \in B(3R): f^k(z)\in B(3R)\ \mbox{for all}\ k \in
\N\right\}\right)\leq t.
\]
Hence $\HD(I_{3R}(f))\leq t$. As
$t>2 M\rho/(2+ M\rho)$
was arbitrary, the   conclusion follows.

\

\section{Lower bounds for the Hausdorff dimension}
\label{lowerbounds}
In order to prove Theorem~\ref{thm2},
we shall use results of Mayer~\cite{Ma} and McMullen~\cite{McM}.
For subsets $A,B$
of the plane (or sphere)
we define the
Euclidean and the spherical  density of $A$ in $B$ by
\[
\dens(A,B)=\frac{\area(A\cap B)}{\area(B)}
\quad\text{and}\quad
\dens_{\chi}(A,B)=\frac{\area_{\chi}(A\cap B)}{\area_{\chi}(B)}.
\]
Note that if
\begin{equation}\label{4m}
B\subset\left\{z\in\C: R<|z|<S\right\},
\end{equation}
then
\[
\frac{4}{(1+S^2)^2}\area(B)
\leq
\area_{\chi}(B)
\leq
\frac{4}{(1+R^2)^2}\area(B)
\]
and thus
\begin{equation}\label{4n}
\left(\frac{1+R^2}{1+S^2}\right)^2
\dens(A,B)
\leq
\dens_{\chi}(A,B)
\leq
\left(\frac{1+S^2}{1+R^2}\right)^2
\dens(A,B)
\end{equation}
if $B$ satisfies~\eqref{4m}

In order to state McMullen's result,
consider for $l\in \N$ a
collection $E_l$ of disjoint compact subsets of $\widehat{\C}$
such that the following two conditions are satisfied:
\begin{enumerate}
\item[(a)] every element of $E_{l+1}$ is
contained in a
unique element of $E_l$;
\item[(b)] every element of $E_l$ contains at least one element of
$E_{l+1}$.\end{enumerate}
Denote by $\overline{E}_l$ the union of all elements of $E_l$ and
put $E=\bigcap^\infty_{l=1} \overline{E}_l$. Suppose that
$(\Delta_l)$ and
$(d_l)$ are sequences of positive real numbers such that if $B\in
E_l$, then
\[\dens_{\chi}(\overline{E}_{l+1},B)\geq\Delta_l\]
and
\[\diam_{\chi}(B)\leq d_l.\]
Then we have the following result~\cite{McM}.

\

\begin{la} \label{lemmamcm}
Let $E$, $E_l$, $\Delta_l$ and $d_l$ be as above. Then
\[\limsup_{l\to \infty}\frac{\sum^{l+1}_{j=1}\;|\log\Delta
_j|}{|\log d_l|}\geq n-\dim E.\]
\end{la}

\

We remark that McMullen worked with the Euclidean density, but the
above lemma follows directly from his result.

We shall use Lemma~\ref{lemmamcm} to prove~\eqref{1c0}.
Of course, it follows from~\eqref{1c0} that
\begin{equation} \label{1c1}
\HD(I_R(f))\geq  \frac{2 M \rho}{ 2+ M \rho},
\end{equation}
for all $R>0$,
but the application of Lemma~\ref{lemmamcm}  does not seem to yield~\eqref{1c},
which says that we have strict inequality in~\eqref{1c1}.
However, in order to illustrate the method, we shall first use Lemma~\ref{lemmamcm}
to prove~\eqref{1c1}. We will then describe the modifications that have to be
made in order to prove~\eqref{1c0}.

The proof of~\eqref{1c} is based on
the following result due to Mayer~\cite{Ma}, which he obtained
using the theory  of infinite iterated function  systems  developed
by Mauldin and Ur\-ba\'nski~\cite{MaU},

\

\begin{la}\label{4.1}
Let $f$  be a transcendental  meromorphic  function  with $ \rho=
\rho(f)< \infty$. Suppose that  $f$  has  a pole $a \in \C \sms
\ov{\Sing(f^{-1})}$ and  denote by $M$ the multiplicity of $a$.
Suppose also that there are a  neighbourhood  $D$ of $a$ and
constants $K
>0$ and $ \a > -1-1/M$ such that $|f'(z)| \le  K |z|^\a $ for  $z
\in f^{-1}(D)$. Then
\begin{equation}\label{4a}
   \HD (J(f))\geq \frac{\rho}{\a+ 1 + 1/M}.
\end{equation}
\end{la}

\

Actually Mayer~\cite[Remark 3.2]{Ma} points out that if $(z_n)$
denotes the sequence  of $a$-points  and if
\begin{equation}\label{4b}
\sum_{n=1}^\infty |z_n|^{-\rho}
\end{equation}
diverges, then we  have  strict inequality in (\ref{4a}). Moreover,
his proof  shows  that if $f$  has infinitely many poles $a$ which
satisfy the   hypothesis  of Lemma~\ref{4.1} and if the series
(\ref{4b}) diverges, then
\begin{equation}\label{4c}
   \HD(I_R(f))>   \frac{\rho}{\a+ 1 + 1/M}
\end{equation}
for each $R>0$.

\section{Construction of the example} \label{construction}
In order to construct a function $f$ to which the results
of the previous section  can be
applied we put $\mu=2/\rho$ and define
\begin{equation}\label{4d}
g(z)=2\sum_{k=1}^{\infty} \frac{k^{\mu k}z^k}{ z^{2k}-k^{2\mu k}}.
\end{equation}
We note   that  if $k \geq \left(2|z|\right)^{1/\mu}$, then
\[
\left|\frac{k^{\mu k}z^k}{ z^{2k}-k^{2\mu k}}\right|\leq
\frac{k^{\mu k}|z|^k}{k^{2\mu k}- |z|^{2k}}\leq 2
\frac{|z|^k}{k^{\mu k}}\leq 2^{1-k}.
\]
Thus the series in (\ref{4d}) converges  locally uniformly and hence
it defines  a function  $g$  meromorphic in $\C$.  The poles  of $g$
are at the points
\[
u_{k,l}=k^{\mu}\exp( \pi  i l /k),
\]
where $ k\in \N$ and  $0\leq  l \leq 2k-1$. With $v_{k,l}=k^{\mu-1}
\exp(\pi i l (1-k)/k)$ we have
\[
g(z)=\sum_{k=1}^\infty \sum_{l=0}^{2k-1}\frac{v_{k,l}}{z-u_{k,l}}.
\]
Note that
\begin{equation}\label{4e}
|v_{k,l}|= k ^{\mu -1}=|u_{k,l}|^{1-1/\mu}=|u_{k,l}|^{1-\rho/2}.
\end{equation}
We will show  that  $g$ is bounded  on the 'spider's web'
$W=W_1\cup W_2$ where
\[
W_1= \bigcup_{n \geq 1} \left\{ z: |z|=\left(n+
\tfrac{1}{2}\right)^{\mu} \right\}
\]
and
\[
W_2=\bigcup_{n \geq 2}
\left\{re^{i \pi (2m-1)/{2n}}: \left(n-\tfrac{1}{2}\right)^\mu \leq r \leq
\left(n+\tfrac{1}{2}\right)^\mu, \, 1\leq m \leq 2n \right\}.
\]
Let first $z\in W_1$, say
$|z|=\left(n+\frac{1}{2}\right)^\mu$ where $n \in \N$.
Then
\[
\begin{aligned}
\frac{1}{2}|g(z)|& \leq \sum_{k=1}^n  \frac{k^{\mu k}|z|^k}{
|z|^{2k}-k^{2\mu k}}+ \sum_{k=n+1}^{\infty} \frac{k^{\mu k}|z|^k}
{k^{2\mu k}- |z|^{2k}}\\
& = \sum_{k=1}^n  \frac{k^{\mu k}}{|z|^{k}-k^{\mu
k}}\frac{|z|^{k}}{|z|^{k}+k^{\mu k}} + \sum_{k=n+1}^\infty
\frac{k^{\mu k}}{k^{\mu k}+|z|^{k} }\frac{|z|^{k}}{k^{\mu k}-|z|^{k}}\\
&\leq \sum_{k=1}^n  \frac{k^{\mu k}}{|z|^{k}-k^{\mu k}} +
\sum_{k=n+1}^\infty \frac{|z|^{k}}{k^{\mu k}-|z|^{ k}}\\
&= \sum_{k=1}^n
\frac{1}{\left(\frac{n+\frac{1}{2}}{k}\right)^{\mu k}-1} +
\sum_{k=n+1}^{\infty} \frac{1}{\left(\frac{k}{n+\frac{1}{2}}
\right)^{\mu k}-1}\\
& = \Sigma_{1,n} + \Sigma_{2,n}
\end{aligned}
\]
Since $\log x \geq (x-1)\log 2 $ for $ 1\leq x \leq 2$ we see that
if $\frac{n}{2}\leq  k \leq n $, then
\[
\left(\frac{n+\frac{1}{2}}{k}\right)^{\mu k}
= \exp \left( \mu k
\log \left(\frac{n+\frac{1}{2}}{k} \right)\right)
\geq \exp\left(\mu k \frac{n+\frac{1}{2}-k}{k}\log 2 \right)
= 2^{\mu (n + \frac{1}{2}-k)}.
\]
With $l=n+1-k $ we deduce that
\begin{equation}\label{4f}
\begin{aligned}
\Sigma_{1,n}& \leq\sum_{k=1}^{\left[\frac{n}{2}\right]}\frac{1}{2^{\mu
k}-1} + \sum_{k=\left[\frac{n}{2}\right]+1}^{n}\frac{1}{2^{\mu(n
+\frac{1}{2}- k)}-1}\\
&=\sum_{k=1}^{\left[\frac{n}{2}\right]}\frac{1}{2^{\mu k}-1} +
\sum_{l=1}^{n-\left[\frac{n}{2}\right]}\frac{1}{2^{\mu(l
-\frac{1}{2})}-1}\\
& \leq \sum_{k=1}^{\infty}\frac{1}{2^{\mu k}-1} +
\sum_{l=1}^{\infty}\frac{1}{2^{\mu(l -\frac{1}{2})}-1}=: C.
\end{aligned}
\end{equation}
Similarly we  obtain
\begin{equation*}
\Sigma_{2,n}\leq
\sum_{k=n+1}^{2n}\frac{1}{\left(\frac{k}{n+\frac{1}{2}}\right)^{ \mu
k}-1}+ \sum_{k=2n+1}^{\infty}\frac{1}{2^{\mu k}-1}.
\end{equation*}
We note that if $n+1 \leq k \leq 2n$, then
\begin{equation*}
\begin{aligned} \left(\frac{k}{n+\frac{1}{2}}\right)^{ \mu k}
&=\exp\left(\mu k \log
\left(\frac{k}{n+\frac{1}{2}}\right)\right)\\
&\geq \exp\left(\mu k \left(\frac{k-n -
\frac{1}{2}}{n+\frac{1}{2}}\right)\log 2\right)\\
&\geq \exp\left( \mu \left( k-n-\tfrac{1}{2}\right)\log 2\right)\\
& = 2^{\mu ( k-n-\frac{1}{2})}.
\end{aligned}
\end{equation*}
With $l=k-n$ we obtain
\begin{equation}\label{4g}
\Sigma_{2,n}\leq \sum_{l=1}^{n}\frac{1}{2^{\mu(l-\frac{1}{2})}-1}+
\sum_{k=2n+1}^{\infty}\frac{1}{2^{\mu k}-1} \leq C.
\end{equation}
Combining (\ref{4f})  with (\ref{4g})  we  find that
\[
|g(z)| \leq 4C  \quad \mbox{for} \quad |z|=\left(n
+\tfrac{1}{2}\right)^\mu.
\]
Let now $z\in W_2$, say
$z=re^{i\pi (2m-1)/2n}$ where $\left(n-\frac{1}{2}\right)^\mu
\leq  r \leq  \left(n+\frac{1}{2}\right)^\mu$ and $1\leq m \leq 2n$.
Then $z^{2n}=-r^{2n}$ and hence
\[
\left|\frac{n^{\mu n}r^n}{ z^{2n}-n^{2\mu n}}\right|
=
\frac{n^{\mu n}r^n}{ r^{2n}+n^{2\mu n}}.
\]
Similar estimates as above now yield
\[ \begin{aligned}
\frac{1}{2}|g(z)|& \leq \sum_{k=1}^{n-1}  \frac{k^{\mu k}r^k}{
r^{2k}-k^{2\mu k}}+  \frac{n^{\mu n}r^n}{ r^{2n}+ n^{2\mu k}}+
\sum_{k=n+1}^{\infty} \frac{k^{\mu k}r^k}{
k^{2\mu k}- r^{2k}}\\
&\leq  \sum_{k=1}^n  \frac{k^{\mu k}}{( n-\frac{1}{2})^{\mu
k}-k^{\mu k}} + 2 + \sum_{k=n+1}^\infty
\frac{(n+\frac{1}{2})^{\mu k}}{k^{\mu k}-(n+\frac{1}{2})^{k} }\\
& = \Sigma_{1,n-1} +2+  \Sigma_{2,n}\\
&\leq 2C +2.
\end{aligned}
\]
We obtain
\begin{equation}\label{4h}
|g(z)|\leq 4C +4 \quad \mbox{for}\quad z\in W.
\end{equation}
Next we want  to show that $g$  is actually bounded on a larger set.
To do this we  note that
\begin{equation}\label{4h1}
\left(n+\frac{1}{2}\right)^\mu- n^\mu \sim n^\mu -
\left(n-\frac{1}{2}\right)^\mu  \sim  \frac{\mu}{2} n^{\mu -1}
\end{equation}
and
\begin{equation}\label{4h2}
|u_{n,m}-u_{n, m+1}|=n^\mu|e^{i\pi /2n}-1|\sim\frac{\pi}{2}
n^{\mu-1}
\end{equation}
as $n \to \infty$. It follows that there exists  $\eta >0$ such that
if $W_{n,m}$ denotes the component  of $\C \sms W$  that contains
$u_{n,m}$, then
\[
\dist(u_{n,m}, \partial{W_{n,m}})\geq 2 \eta n ^{\mu-1}
\]
for all $n \in \N$  and $m \in \{0, 1, \ldots, 2n-1\}$. The function
\[
h(z)=g(z)- \frac{v_{n,m}}{z-u_{n,m}}
\]
is holomorphic  in the closure of $W_{n,m}$ and  for $z \in
\partial{W_{n,m}}$ we have
\begin{equation*}\label{4i}
|h(z)|
\leq |g(z)| + \frac{v_{n,m}}{|z-u_{n,m}|}
\leq  4C+4 + \frac{n^{\mu-1}}{2 \eta n^{\mu-1}}
= 4C+ 4 +\frac{1}{2\eta}.
\end{equation*}
By  the maximum principle,
\begin{equation*}\label{4j}
|h(z)|\leq 4 C+ 4+ \frac{1}{2\eta}\quad \mbox{for} \quad  z \in
W_{n,m}.
\end{equation*}
We put $r_n=\eta n^{\mu-1}$ and  deduce  that if $z\in W_{n,m}\sms
D(u_{n,m},r_n)$, then
\[
|g(z)|\leq |h(z)|+ \frac{|v_{n,m}|}{r_n}\leq 4C + 4 +
\frac{3}{2\eta}.
\]
In order  to show that $g \in \mathcal  B$ we note  that if $z \in
\partial{W_{n,m}}$, then
\begin{equation*}
|g'(z)|
= \frac{1}{2\pi}\left|\;
\int\limits_{|\zeta-z|=r_n}\frac{g(\zeta)}{(\zeta-z)^2}d\zeta \right|
\leq \frac{1}{r_n} \max_{|\zeta-z|=r_n}|g(z)|
\leq \frac{1}{r_n}\left( 4C+ 4+ \frac{3}{2\eta}\right).
\end{equation*}
This implies  that if $z\in \partial{W_{n,m}}$, then
\[
|h'(z)|\leq |g'(z)|+ \frac{|v_{n.m}|}{|z-u_{n,m}|^2}\leq
\frac{1}{r_n}\left( 4C+ 4+ \frac{5}{2\eta}\right).
\]
Again we have
\[
|h'(z)|\leq \frac{1}{r_n}\left( 4C+ 4+ \frac{5}{2\eta}\right) \quad
\mbox{for} \quad z\in W_{n,m}
\]
by the maximum principle. We deduce  that if $\d >0$ is chosen sufficiently
small and $z\in D(u_{n,m},\d r_n)$, then
\[
|g'(z)|
\geq \frac{|v_{n,m}|}{|z-u_{n,m}|^2}-|h'(z)|
\geq
\frac{1}{r_n}\left(\frac{1}{\d^2 \eta}- 4 C -4
-\frac{5}{2\eta}\right)
>0.
\]
It follows  that if $ g'(z)=0$ for some  $z\in W_{n,m}$, then
$|z-u_{n,m}| \geq \d r_n$ and thus
\[
|g(z)| \leq |h(z)| + \frac{|v_{n,m}|}{|z-u_{n,m}|} \leq 4 C +4 +
\frac{1}{2\eta} +\frac{1}{\d \eta}.
\]
This implies  that the set of critical  values of $g$ is   bounded.
By (\ref{4h}) the same is  true for  the set of asymptotic values of
$g$. Hence $g \in \mathcal B$.

To compute the order of $g$ we note that the number $n(r,g)$ of
poles of $g$ in $\overline{D(0,r)}$ satisfies
\[
n(r,g)=\sum_{j=1}^{\left[r^{1/\mu}\right]} 2k
\sim \int\limits_{0}^{r^{1/\mu}} 2t\; dt=r^{2/\mu}=r^\rho
\]
as $r\to\infty$. Thus
\[
N(r,g)=\int\limits_0^r \frac{n(t,g)}{t} dt \sim \frac{1}{\rho} r^\rho.
\]
By (\ref{4h}) we have $m(r,g)\leq 4C+4$ if $r$ has the form
$r=\left(k+\frac12\right)^\mu$ for some $k\in\N$.
It follows that
\begin{equation}\label{4j1}
T(r,g)=N(r,g)+m(r,g)\sim \frac{1}{\rho} r^\rho
\end{equation}
as $r\to\infty$ through $r$-values of the form
$r=\left(k+\frac12\right)^\mu$.
But since $T(r,g)$ is an increasing function of~$r$,
the relation~(\ref{4j1}) actually holds for all values of~$r$.
Hence $g$ has order~$\rho$.

\

We now put
\[
f(z)=g(z)^M.
\]
It follows  that $f \in
\mathcal B$ and that $f$ has order~$\rho$.

\
\section{Proof of Theorem~\ref{thm2}}\label{proofthm2}
Let $f$ be the function constructed in section~\ref{construction}.
As in section~\ref{lowerbounds} we denote the sequence of poles
by $(a_j)$, ordered such that $|a_j|\leq |a_{j+1}|$ for all
$j\in\N$.
For each $j\in\N$ we thus have $a_j=u_{n,m}$ for some $n\in\N$
and $0\leq m\leq 2n-1$. It is not difficult to
see that $n\sim j^2$ as $j\to\infty$ if $a_j=u_{n,m}$.
Hence $|a_j|=|u_{n,m}|=n^\mu\sim j^{2\mu}= j^{1/\rho}$ as $j\to\infty$. 
Choose $b_j$ as in
section~\ref{proofthm1} so that
\[
f(z)\sim\left(\frac{b_j}{z-a_j}\right)^{M} \quad \mbox{as}\quad
z\to a_j.
\]
Then  $b_j=v_{n,m}$ if  $a_j=u_{n,m}$ and hence
\begin{equation}\label{4y1}
|b_j|=|a_j|^{1-\rho/2}
\end{equation}
by~\eqref{4e}.
Choose $R>0$ large and let $E_l$ be as in section~\ref{proofthm1}.
Thus $E_l$ consists of all components $V$ of $f^{-l}(B(R))$ for
which $f^k(V)\subset B(R)$ for $0\leq k\leq l-1$.
Clearly
$E=\bigcap_{l=1}^\infty\overline{E}_l \subset I_R(f)$.

We deduce from~\eqref{3i} that if $V\in E_l$ and
$f^k(V)\subset U_{j_{k+1}}$ for $0\leq k\leq l-1$, then
\[
\diam_{\chi} \left(V\right)
\leq
(2^{1/M} 24)^{l-1} \frac{32}{R^{1/M}} \prod_{k=1}^l
\frac{|b_{j_k}|}{|a_{j_k}|^{1+1/M}}.
\]
By~\eqref{4y1} we have
\[
\frac{|b_{j_k}|}{|a_{j_k}|^{1+1/M}}=\frac{1}{|a_{j_k}|^{\rho/2+1/M}}
\]
and since $|a_{j_k}|>R$ we obtain
\[
\diam_{\chi} \left(V\right)
\leq \left( \frac{A}{R^{\rho/2+1/M}}\right)^l
\]
for some constant $A>0$ if $V\in E_l$.
Thus we can apply Lemma~\ref{lemmamcm} with
\begin{equation}\label{7a}
d_l=\left( \frac{A}{R^{\rho/2+1/M}}\right)^M.
\end{equation}

In order to estimate $\Delta_l$ we note that
\[
W_{n,m}\subset D\left(u_{n,m},\left(\frac{\mu}{2}+\frac{\pi}{4}\right) n^{\mu-1}\right)
\]
for large $n$ by~\eqref{4h1} and~\eqref{4h2}. By~\eqref{4e} we have
\[
|v_{n,m}|=|u_{n,m}|^{1-\rho/2}=n^{\mu(1-\rho/2)}=n^{\mu-1}.
\]
With
$\tau=\mu/2+\pi/4$ 
we see that
$W_{n,m}\subset D\left(u_{n,m},\tau|v_{n,m}|\right)$
if $n$ is large. Thus
\[
W_{n,m}\subset D\left(a_j,\tau|b_j|\right)
\]
if  $a_j=u_{n,m}$.
On the other hand, it follows from~\eqref{3a} and~\eqref{4h} that
\[
D\left(a_j,\frac{1}{4 R^{1/M}}|b_j|\right)\subset U_j=W_{n,m}\cap f^{-1}(B(R)),
\]
provided $R$ is large enough. We conclude that
\[
\dens\left( f^{-1}(B(R)),W_{n,m}\right)\geq \frac{1}{16\tau^2 R^{2/M}}.
\]
Since $W_{n,m}\subset \left\{ z\in\C: \left(n-\frac12\right)^\mu
\leq |z|\leq  \left(n+\frac12\right)^\mu\right\}$ and since
$\left(n+\frac12\right)^\mu/\left(n-\frac12\right)^\mu\to 1$
as $n\to\infty$ this implies that
if $S\geq R$ and
\[A(S)=\left\{ z\in\C: S< |z|<  2S\right\},
\]
then
\[
\dens\left( \overline{E}_1,A(S)\right)\geq \frac{1}{17\tau^2 R^{2/M}}.
\]
We now consider a branch $g_j$ of  $f^{-1}$ which maps
$A'(S)=A(S)\setminus (-2S,-S)$ into $U_j$.
Recall that $g_j$ has the form~\eqref{3g}.
It follows from~\eqref{3g1} that
\[
|g_j'(z)|\leq \frac{12|b_j|}{M S^{1+1/M}}
\]
for $z\in A'(S)$. The argument to obtain~\eqref{3g1} also shows that
\[
|g_j'(z)|
\geq \frac{4|b_j|}{27 M |z|^{1+1/M}}
\geq \frac{4|b_j|}{27 M (2S)^{1+1/M}}
\]
for $z\in A'(S)$.
With $K=2^{1+1/M}81$ we obtain
\[
\sup_{u,v\in A'(S)}\left|\frac{g_j'(u)}{g_j'(v)}\right|\leq K,
\]
provided $S$ is large enough.
We deduce that
\[
\dens\left( g_j\left(\overline{E}_1\right),
g_j(A'(S))\right)\geq
\frac{1}{K^2}\dens\left( \overline{E}_1,
A'(S)\right)\geq
\frac{1}{17K^2\tau^2 R^{2/M}}.
\]
Applying this for all $S$ for the form $S=2^k R$ with $k\geq 0$ and
for all branches $g_j$ mapping to $U_j$ we deduce that
\begin{equation}\label{7b}
\dens\left( \overline{E}_2, U_j\right)\geq
\frac{1}{17K^2\tau^2 R^{2/M}}
\end{equation}
for each element $U_j$ of $E_1$.
Let now $V\in E_l$ and $j_1,j_2,\ldots,j_k$ such that
$f^k(V)\subset U_{j_{k+1}}$ for $0\leq k\leq l-1$. Then
$f^{l-1}(V)= U_{j_{l}}$ and
\begin{equation}\label{7c}
f^{l-1}\left( \overline{E}_{l+1}\cap V\right)= \overline{E}_2\cap U_{j_{l}}.
\end{equation}
For large $R$ a branch of $f^{-1}$ that maps $U_{j_l}$ into
$U_{j_{l-1}}$ extends univalently to $D\left(a_{j_{l}},\frac34 a_{j_{l}}\right)$
and it maps $D\left(a_{j_{l}},\frac34 a_{j_{l}}\right)$ into $B(R)$. Thus
the branch of the inverse of $f^{l-1}$ which maps $U_{j_{l}}$ to $V$
extends univalently do $D\left(a_{j_{l}},\frac34 a_{j_{l}}\right)$.
Since $U_{j_{l}}\subset D\left(a_{j_{l}},\frac12 a_{j_{l}}\right)$
by~\eqref{3e}, we can now deduce from~\eqref{7b}, \eqref{7c} and
Koebe's distortion theorem~\eqref{2b} with $\lambda=\frac34$ that
\[
\dens\left( \overline{E}_{l+1}, V\right)
\geq
\left(\frac{1-\lambda}{1+\lambda}\right)^4
\dens\left( \overline{E}_2, U_{j_l}\right)
\geq
\left(\frac{1-\lambda}{1+\lambda}\right)^4
\frac{1}{17K^2\tau^2 R^{2/M}}.
\]
Since $U_{j_{l}}\subset D\left(a_{j_{l}},\frac12 a_{j_{l}}\right)$
we conclude using~\eqref{4n} that there exists a constant $B>0$ such that
\[
\dens_{\chi}\left( \overline{E}_{l+1}, V\right)
\geq \frac{B}{R^{2/M}}.
\]
Hence Lemma~\ref{lemmamcm} can be applied with
\begin{equation}\label{7d}
\Delta_l=\frac{B}{R^{2/M}}.
\end{equation}
Using the values for $d_l$ and $\Delta_l$ given by~\eqref{7a}
and~\eqref{7d} we find that
\[
\HD(E)\geq 2-\limsup_{l\to\infty}
\frac{(l+1)\left(\log B - \frac{2}{M}\log R\right)}{l\left(
\log A-\left(\frac{\rho}{2}+\frac{1}{M}\right)\log R\right)}
=2-
\frac{\log B - \frac{2}{M}\log R}{\log A-
\left(\frac{\rho}{2}+\frac{1}{M}\right)\log R}.
\]
Since $E\subset I_R(f)$ and thus $\HD\left(I_R(f)\right)\geq \HD(E)$
and since $\HD\left(I_R(f)\right)$ is an non-increasing function of $R$,
we obtain
\[
\HD\left(I_R(f)\right)
\geq 2-\limsup_{R\to\infty}
\frac{\log B - \frac{2}{M}\log R}{\log A-
\left(\frac{\rho}{2}+\frac{1}{M}\right)\log R}
=
2-\frac{\frac{2}{M}}{\frac{\rho}{2}+\frac{1}{M}}
=\frac{2M\rho}{2+M\rho}.
\]
Thus we have proved~\eqref{1c1}.

\

In order to prove~\eqref{1c0}  we choose a non-decreasing
sequence $(R_l)$ which tends to~$\infty$.
We define $E_l$ as the set of all components of $f^{-l}(B(R_l))$ for
which $f^k(V)\subset B(R_{l-k})$ for $0\leq k\leq l-1$.
Then $E=\bigcap_{l=1}^\infty\overline{E}_l \subset I(f)$. The same considerations as
before now yield that we can apply Lemma~\ref{lemmamcm} with
\[
d_l=A^l \prod_{k=1}^l \frac{1}{R_k^{\rho/2+1/M}}
\]
and
\[
\Delta_l=\frac{B}{R_l^{2/M}}.
\]
We obtain
\[
\HD(E)\geq 2-\limsup_{l\to\infty}
\frac{(l+1)\log B - \frac{2}{M}\sum_{k=1}^{l+1}\log R_k}{l
\log A-\left(\frac{\rho}{2}+\frac{1}{M}\right)\sum_{k=1}^{l}\log R_k}.
\]
Choosing a sequence $(R_l)$ which does not tend to infinity too fast,
for example $R_k=k$ for large $k$, we deduce that
\[
\HD\left(I(f)\right)
\geq 2-\frac{\frac{2}{M}}{\frac{\rho}{2}+\frac{1}{M}}
=\frac{2M\rho}{2+M\rho}.
\]
The opposite inequality follows from Theorem~\ref{thm1}.
Thus we have proved~\eqref{1c0}.

\

To prove~\eqref{1c} we will now
apply  Lemma~\ref{4.1} and the remarks following it.
Let $a$ be a pole  of $f$ which has large  modulus.
Thus $a=u_{n,m}$ where $n$ is large and $0 \leq  m \leq 2n-1$. It
follows from the consideration in section~3  that if $a$ is large
enough, if $D$  is a sufficiently  small neighbourhood of $a$ and if
$z\in f^{-1}(D), $  then   $z$ is  in a  small  neighbourhood  of
one of the poles $u_{k,l}$. In particular, we can achieve  that
\begin{equation}\label{4k}
\frac{1}{2}\frac{|v_{k,l}|^M}{|z-u_{k,l}|^M}\leq |f(z)| \leq 2 |a|
\end{equation}
and
\begin{equation}\label{4l}
|f'(z)| \leq 2 M \frac{|v_{k,l}|^M}{|z-u_{k,l}|^{M+1}}
\end{equation}
for   some $k\in \N$ and $0\leq  l \leq 2k-1$, if $|z|$  is
sufficiently  large. Combining (\ref{4k}) and  (\ref{4e})  we see
that
\begin{equation}\label{4l1}
|z-u_{k,l}|\geq \left(\frac{1}{4|a|}\right)^{1/M}|v_{k,l}|=
\left(\frac{1}{4|a|}\right)^{1/M}|u_{k,l}|^{1-\rho/2}.
\end{equation}
Now (\ref{4k}), (\ref{4l}) and (\ref{4l1})  yield
\[
|f'(z)|
\leq 2 M
\frac{|v_{k,l}|^M}{|z-u_{k,l}|^M}\frac{1}{|z-u_{k,l}|}
\leq  \frac{8 M |a|}{|z-u_{k,l}|}
\leq 8 M|a|( 4|a|)^{1/M}|u_{k,l}|^{\rho/2-1},
\]
and as $z$  is in a  small neighbourhood of $u_{k,l}$ we obtain
\[|f'(z)|\leq K |z|^{\rho/2-1}
\]
for some constant $K$. We can  thus apply Mayer's result with
$\a=\rho/2-1$
and hence (\ref{4a})  yields
\[
\HD(J(f))\geq \frac{\rho}{\frac{\rho}{2}+\frac{1}{M}}= \frac{ 2 M\rho}{2+M \rho}.
\]
Thus we have again obtained~\eqref{1c1}.

However, from~\eqref{4j1} and the definition of $f$ we deduce that
\[
T(r,f)\sim \frac{M}{\rho}r^\rho
\]
as $r\to\infty$. This implies that
the series (\ref{4b}) diverges for all $a\in\C$ with at most two exceptions.
Hence (\ref{1c})  follows from (\ref{4c}).

\

\section{Proof of Theorem~\ref{thm3}}
Suppose that $\area\left(I_R(f)\right)>0$.
Putting $I_R'=\{z\in\C:|f^k(z)|>R\text{ for all }k\geq 0\}$ we
have $\area(I_R')>0$. We shall show that this leads to a
contradiction if $R$ is sufficiently large.

We use the notation of section~\ref{proofthm1} and,
in addition,
denote by $U_j^0$ the component of $B(R_0)$ that
contains $a_j$. Then $U_j^0\cap U_k^0=\emptyset$
for $j\neq k$. In particular, $U_j^0\cap U_k=\emptyset$
for $j\neq k$ if $R\geq R_0$.
It follows that $I_R'\cap U_j^0\subset U_j$.
By~\eqref{3a} we have
\[
U_j^0\supset D\left(
a_j,\frac{1}{4 R_0^{1/M}} |b_j| \right)
\]
while~\eqref{3c} yields
that
\[
U_j\subset  D\left( a_j,\frac{2}{ R^{1/M}} |b_j| \right).
\]

Let now  $\xi$ be a density point of $I_R'$ and put $w_l=f^l(\xi)$
for $l\in\N$. Then $w_l\in U_{j_l}$ for some $j_l\in\N$.
Since $|w_l-a_{j_l}|\leq 2R^{-1/M}$ we have
\[
D\left(w_l,\frac{1}{5 R_0^{1/M}} |b_{j_l}| \right)
\subset
D\left(a_{j_l},\frac{1}{4 R_0^{1/M}} |b_{j_l}| \right)
\]
for large $R$.
Thus
\[
I_R'\cap D\left(w_l,\frac{1}{5 R_0^{1/M}} |b_{j_l}| \right)
\subset D\left( a_{j_l},\frac{2}{R^{1/M}} |b_{j_l}| \right)
\]
which implies that
\[
\dens\left(I_R',D\left(w_l,\frac{1}{5 R_0^{1/M}} |b_{j_l}| \right)\right)
\leq
100  \left(\frac{R_0}{R}\right)^{2/M}
\]
Similarly as in section~\ref{proofthm2} we see that if $R_0$
is chosen large enough and if $\phi$ denotes
the branch of the inverse function of $f^l$ which maps
$w_l$ to $\xi$, then $\phi$ has an analytic continuation
to $D(w_l,\frac{2}{5}R_0^{-1/M}|b_{j_l}|)$.
Applying Koebe's distortion theorem with $\lambda=\frac12$
we conclude that
\[
\dens\left(I_R',\phi\left(D\left(w_l,\frac{1}{5 R_0^{1/M}} |b_{j_l}|
\right)\right)\right)
\leq
8100  \left(\frac{R_0}{R}\right)^{2/M}.
\]
Koebe's theorem also yields that
\[
D\left(\xi,\frac{1}{20 R_0^{1/M}} |b_{j_l}\phi'(w_l)| \right)
\subset
\phi\left(D\left(w_l,\frac{1}{5 R_0^{1/M}} |b_{j_l}| \right)\right)
\subset
D\left(\xi,\frac{2}{5 R_0^{1/M}} |b_{j_l}\phi'(w_l)|\right).
\]
With $r_l=\frac{2}{5} R_0^{-1/M} |b_{j_l}\phi'(w_l)|$ we conclude
that
\begin{equation}\label{7f}
\dens\left(I_R',\left(D\left(\xi,r_l \right)\right)\right)
\leq
64\cdot 8100  \left(\frac{R_0}{R}\right)^{2/M}.
\end{equation}
Also, it is not difficult to see that $r_l\to 0$ as $l\to\infty$.
If $R$ is so large that the right hand side of~\eqref{7f}
is less than~$1$,
we obtain a contradiction to the
assumption that $\xi$ is a point of density.\qed

\

We note that the argument shows that in fact the set
of all $z\in\C$ for which
\[
\limsup_{k\to\infty}|f^k(z)|>R
\]
has area zero for large~$R$.

\section{Proof of Theorem~\ref{thm4}}

We want to  construct a  function $f \in \mathcal B$ for which
$\infty$ is not an  asymptotic value, the multiplicity  of the poles
is unbounded and $ \area(I(f))>0$.

\

We begin by choosing a sequence of discs $D(a_j, r_j)$
of radius less than $1$ which are contained
in $\left\{z\in {\mathbb C}:  |z| > 2\right\}$  such that
the complement
\[
A={\mathbb C}\sms\bigcup_{j=1}^\infty D(a_j, r_j)
\]
is small in a certain sense. More specifically, we choose
the $D(a_j, r_j)$ such that with
\[
P_n=\left\{z\in {\mathbb C}: 2^{n}\leq |z|< 2^{n+1}\right\}.
\]
the following properties are satisfied:
\[
\ov{D(a_j, r_j)}\cap \ov{ D(a_k, r_k)}=\es \quad \mbox{for}\ j,k\in\N,
j\neq k,
\]
\[
I_n:= \left\{ j\in\N: P_n\cap D\left(a_j,r_j\right)\neq\emptyset\right\}
\  \mbox{is finite for}\ n\in\N
\]
and
\[
\area\left(A\cap P_n\right) < 1 \quad \mbox{for}\ n\in\N.
\]
It is clear that it is possible to choose a sequence
of disks with these properties.

Next we choose a sequence $(r_k')$ satisfying
$0<r_k'<r_k$ for all $k\in\N$ such that with
\[
A'={\mathbb C}\sms\bigcup_{j=1}^\infty D(a_j, r_j')
\]
we have
\begin{equation}\label{5a}
\area\left(A'\cap P_n\right) < 2
\end{equation}
for all $n\in\N$.
For $k\in\N$ we put
\[
 d_k=\min_{j\neq k}  \dist(a_k, D(a_j,r_j)).
\]
Note that $d_k>r_k$ for all $k\in\N$.
We also choose a sequence $(\eps_k)$ of  positive real numbers
such that
\begin{eqnarray}\label{5e}
\sum_{k=1}^{\infty}\eps_k  <
\frac{1}{2}.
\end{eqnarray}
Finally we choose
a sequence $(m_k)$ of positive integers
such that
\begin{equation}\label{5f}
\frac{\eps_k m_k}{r_k}> 2,
\end{equation}
\begin{equation}\label{5g}
\eps_k\left(\frac{r_k}{r_k'}\right)^{m_k}> 3
\end{equation}
and
\begin{equation}\label{5h}
\frac{ m_k}{d_k}\left(\frac{r_k}{d_k}\right)^{m_k}\leq 1
\end{equation}
for all $k\in\N$ and
\begin{equation}\label{5i}
\sum_{k\in I_n}
\frac{r_k^2}{m_k}\leq \frac{3}{32}
\end{equation}
for all $n\in\N$.
The function $ f:{\mathbb C}\to \widehat{\mathbb C}$ is now
defined by
\[
f(z)=\sum_{k=1}^{\infty}\eps_k
\left(\frac{r_k}{z-a_k}\right)^{m_k},
\]

\

\begin{la}\label{5.1}
The function $f$ is in ${\mathcal B}$ and
$\infty$ is not an asymptotic value  of $f$.
\end{la}

\begin{proof}
The derivative of $f$ is  given by
\[
f'(z)=- \sum_{k=1}^\infty
\frac{\eps_km_k}{z-a_k}\left(\frac{r_k}{z-a_k}\right)^{m_j}.
\]
For $z\in D(a_k, r_k)\sms\{a_k\}$ we thus  have
\[
|f'(z)|\geq
\frac{\eps_km_k}{|z-a_k|}\left(\frac{r_k}{|z-a_k|}\right)^{m_k}-
\sum_{\substack{j=0 \\ j\neq k}}^\infty \frac{\eps_j
m_j}{|z-a_j|}\left(\frac{r_j}{|z-a_j|}\right)^{m_j}
\]
and hence, using the definition of $d_j$,
\begin{equation}\label{5m}
|f'(z)|\geq
\frac{\eps_km_k}{r_k}\left(\frac{r_k}{|z-a_k|}\right)^{m_k}-
 \sum_{\substack{j=0 \\ j\neq k}}^\infty \frac{\eps_j
m_j}{d_j}\left(\frac{r_j}{d_j}\right)^{m_j}.
\end{equation}
For $z \in D(a_k, r_k)$ we have $|z-a_k|< r_k$ and thus~(\ref{5f}) yields
\begin{equation}\label{5n}
\frac{\eps_km_k}{r_k}\left(\frac{r_k}{|z-a_k|}\right)^{m_k}\geq
\frac{\eps_km_k}{r_k}>  2.
\end{equation}
On the other hand, applying  (\ref{5h}) and (\ref{5e})  we obtain
\begin{equation}\label{5o}
\sum_{\substack{j=0 \\ j\neq k}}^\infty \frac{\eps_j
m_j}{d_j}\left(\frac{r_j}{d_j}\right)^{m_j} \leq \sum_{j=1}^\infty
\eps_j < 1.
\end{equation}
It follows from (\ref{5m}), (\ref{5n}) and (\ref{5o}) that
\begin{equation}\label{5o1}
|f'(z)|\geq
\frac{1}{2}\frac{\eps_km_k}{r_k}\left(\frac{r_k}{|z-a_k|}\right)^{m_k}
>1
\end{equation}
for all $z\in D(a_k, r_k)\sms \{a_k\}$. The last   inequality implies
that all critical  points of $f$  are contained in $A$.
Since
\begin{equation}\label{5p1}
|f(z)|\leq \sum_{k=1}^\infty \eps_k< \frac12
\quad \mbox{for}\quad z\in A,
\end{equation}
all critical and asymptotic values
of $f$  are contained in $D\left(0,\frac12\right)$.
Hence $f \in \mathcal B$ and  $\infty$ is not an asymptotic value of $f$.
\end{proof}
We also note that
$D(0,2)\sbt A$ so that (\ref{5p1}) yields
that $f(D(0,2)) \sbt D\left(0,\frac12\right)$.
Hence $D(0,2)$ contains
an attracting fixed point and  all singular values are contained in
its basin of attraction.

\begin{la}\label{5.1a}
If $z\in\bigcup_{j=1}^\infty D(a_j,r_j')=\C\setminus A'$, then
$|f(z)|>2$.
\end{la}
\begin{proof}
Let $z \in D(a_k, r_k')$. Proceeding as in
the proof of Lemma~\ref{5.1},  we obtain
\[
|f(z)|
\geq
\eps_k\left(\frac{r_k}{|z-a_k|}\right)^{m_k}-
\sum_{\substack{j=0 \\ j\neq k}}^\infty
\eps_j \left(\frac{r_j}{|z-a_j|}\right)^{m_k}
\geq
\eps_k\left(\frac{r_k}{r_k'}\right)^{m_k}-
\sum_{j=0}^\infty\eps_j .
\]
The conclusion now follows from~\eqref{5e} and (\ref{5g}).
\end{proof}
Since $\Sing\left(f^{-1}\right)\subset D\left(0,\frac12\right)$ we can define the
branches of
the inverse function of $f$ in every simply-connected domain contained in
$\C\setminus \overline{D(0, 1)}$. Because of~\eqref{5p1} such a branch of
$f^{-1}$ maps this domain into $D(a_k,r_k)$ for some $k\in\N$.

\begin{la}\label{5.2}
Let $g: \{w\in {\mathbb C}:
|w|>1\} \sms (-\infty, -1)\to  D(a_k,r_k)$ be a  branch of $f^{-1}$. Then
\begin{equation*}\label{5r}
|g'(w)|\leq \frac{4r_k}{m_k|w|}
\end{equation*}
for $|w|>1$.
\end{la}
\begin{proof}  For $|w|> 1$ we have
\begin{equation}\label{5s}
|g'(w)|=\frac{1}{|f'(g(w))|}\leq \frac{2r_k}{\eps_km_k}
\left(\frac{|g(w)-a_k|}{r_k}\right)^{m_k}
\end{equation}
by~\eqref{5o1}.
If $z \in D(a_k, r_k)$ and $|f(z)|>1$ then by (\ref{5e})
\[
\begin{aligned}
|f(z)|&\leq  \eps_k\left(\frac{r_k}{|z-a_k|}\right)^{m_k} +
\sum_{\substack{j=0 \\ j\neq k}}^\infty \eps_j
\left(\frac{r_j}{|z-a_j|}\right)^{m_j}\\
& \leq \eps_k\left(\frac{r_k}{|z-a_k|}\right)^{m_k}+
\sum_{\substack{j=0 \\ j\neq k}}^\infty \eps_j\\
& \leq \eps_k\left(\frac{r_k}{|z-a_k|}\right)^{m_k} +
\frac{1}{2}|f(z)|.
\end{aligned}
\]
Thus
\[
|f(z)|\leq 2 \eps_k\left(\frac{r_k}{|z-a_k|}\right)^{m_k}
\]
and with $w=f(z)$ we obtain
\begin{equation}\label{5t}
|w|\leq  2 \eps_k\left(\frac{r_k}{|g(w)-a_k|}\right)^{m_k}.
\end{equation}
It follows from (\ref{5s}) and (\ref{5t}) that
\[
|g'(w)|\leq \frac{2 r_k}{\eps_km_k}\cdot
\frac{2\eps_k}{|w|}=\frac{4r_k}{m_k |w|}.
\]
\end{proof}

 We put $A^*=\{z\in A:|z|>2\}$.

\begin{la}\label{5.3}
 For $k,n\in\N$ we have
\[
\area\left(f^{-k}(A^*)\cap P_n\right)\leq \frac{1}{2^k}.
\]
\end{la}

\begin{proof} Let $g: \{w\in {\mathbb
C}: |w|>1\} \sms (-\infty, -1)\to  D(a_k,r_k)$  be a  branch of
$f^{-1}$. By Lemma~\ref{5.2} we
have
\[\begin{aligned} \iint_{A^*}|g'(w)|^2 dxdy &
=\sum_{l=1}^\infty\iint_{A\cap P_l}|g'(w)|^2 dxdy \\
& \leq \sum_{l=1}^\infty  \area\left(A\cap P_l\right)\cdot \left(
\frac{4r_k}{m_k }\right)^2\cdot \frac{1}{(2^l)^2}\\
& \leq\frac{16 r_k^2}{m_k^2}\sup_{l\geq 1} \area\left(A\cap P_l\right)
\sum_{l=1}^\infty
\frac{1}{2^{2l}}\\
& =\frac{16}{3}\frac{r_k^2}{m_k^2} \sup_{l\geq 1} \area\left(A\cap P_l\right)
\end{aligned}
\]
Noting that there are $m_k$ such branches of $f^{-1}$ we
deduce from~\eqref{5i} that
\[
\area\left(f^{-1}(A^*)\cap P_n\right)
\leq
\sup_{l\geq 1}  \area\left(A\cap P_l\right)
\sum_{k\in I_n}
\frac{r_k^2}{m_k}
\leq \frac12 \sup_{l\geq 1}  \area\left(A\cap P_l\right)
\]
for all $n\in\N$.
Analogously we find that
\[
\area\left(f^{-2}(A^*)\cap P_n\right)
\leq \frac12 \sup_{l\geq 1} \area\left(f^{-1}(A^*)\cap P_l\right)
\]
and induction yields that
\[
\area\left(f^{-k}(A^*)\cap P_n\right)\leq \frac{1}{2^k}
\sup_{l\geq n_0} \area\left(A\cap P_l\right)\leq \frac{1}{2^k}.
\]
\end{proof}

Put $B={\mathbb C}\sms \bigcup_{k=0}^\infty f^{-k}(A)$.
Then
\begin{equation}\label{5v}
\C\setminus B=\bigcup_{k=0}^\infty f^{-k}(A)
\subset A'\cup \bigcup_{k=0}^\infty f^{-k}(A^*).
\end{equation}

\

\begin{la}\label{5.4}
 $\area(B)>0$.
\end{la}
\begin{proof} It follows from Lemma~\ref{5.3} that
\[
\area\left(\bigcup_{k=1}^{\infty} f^{-k}(A^*)\cap P_n\right)
\leq
\sum_{k=0}^{\infty}\frac{1}{2^{k}}=2.
\]
By (\ref{5a}) we have
\[
\area\left(A'\cap P_n\right) <2.
\]
Thus \eqref{5v} yields that
$\area\left(P_n\setminus B\right) <4$ and the conclusion follows.
\end{proof}

\

\begin{la}\label{5.5}
$\area(B\sms I(f))=0$.
\end{la}
\begin{proof} Suppose that $\area(B\sms I(f))>0$ and
let $\xi$ be a density point of $B\sms I(f)$.
Since $\xi\in B$ we have $f^m(\xi)\in {\mathbb C} \sms A$ and thus
in particular $|f^m(\xi)| > 2$  for $m\in \mathbb N$.
As $\xi \notin I(f)$ there is a sequence $(m_l)$ tending to
$\infty$ and a constant $R>0$
such that  $ |f^{m_l}(\xi)| \leq R$. Put $w_l=f^{m_l}(\xi)$.
Passing to a subsequence if necessary we may assume that
$w_l\to w$ where $2\leq |w|\leq R$.
Since the disks $D(a_j,r_j)$ have radius less than $1$ we have
\[
\a:=\area(D(w,1)\cap A) >0.
\]
Since
$\Sing\left(f^{-1}\right)\subset D\left(0,\frac12\right)$
and since $f\left(D\left(0,\frac12\right)\right)\subset
f(A)\subset D\left(0,\frac12\right)$ the branch
$g_{l}$ of $f^{-m_l}$ which maps $w_l$ onto $\xi$
exists as a univalent function in $D\left(w,\frac32\right)$.
Fix $\delta$ with $0<\delta<\frac14$ and choose $l$ so large that $|w_l-w|\leq \delta$.
Put $D_l=D\left(w_l,1+\delta\right)$. Then $D(w,1)\subset D_l\subset D\left(w,1+2\delta\right)$
and Lemma~\ref{2.1} yields with $\lambda=2(1+2\delta)/3$ that
\[
\sup_{u,v \in  D_l}
\frac{|g_l'(u)|}{|g_l'(v)|}\leq K:=\left(\frac{1+\lambda}{1-\lambda}\right)^4.
\]
We obtain
\[
\dens(g_{l}(A),g_{l}(D_l))\geq \frac{1}{K^2}
\dens(A,D_l)\geq \frac{1}{K^2}
\dens(D(w,1)\cap A,D_l)=\frac{\alpha}{K^2\pi(1+\delta)^2}.
\]
Lemma~\ref{2.1} also yields that there exists constants $\gamma_1,
\gamma_2>0$ such that
\[
D\left(\xi,\gamma_1 |g_l'(w_l)|\right)\subset g_{l}(D_l)\subset
D\left(\xi,\gamma_2 |g_l'(w_l)|\right).
\]
With $r_l=\gamma_2 |g_l'(w_l)|$ it follows that
\[
\dens\left(g_{l}(A),D\left(\xi,r_l\right)\right)
>\frac{\gamma_1^2\alpha}{\gamma_2^2 K^2\pi(1+\delta)^2}.
\]
Also, (\ref{5o1}) shows that
$g_{m_l}'(w_l) \to 0$ and hence $r_l\to 0$ as $l\to\infty$.
Since $ g_{l}(A)\cap B=\emptyset$ for all $l\in\N$
this contradicts
the assumption that $\xi$ is a density point of $B$.
\end{proof}


Theorem~\ref{thm4} follows from Lemmas~\ref{5.4} and~\ref{5.5}.


\begin{thebibliography}{99}

\bibitem{Bar}
K.\ Bara\'nski, Hausdorff dimension of hairs and ends for entire maps
of finite order.  Math. Proc. Cambridge Philos. Soc.
145 (2008), 719--737.
\bibitem{BKZ}
K.\ Bara\'nski, B.\ Karpi\'nska and A.\ Zdunik, Hyperbolic dimension
of Julia sets of meromorphic maps with logarithmic tracts.
Int.\ Math.\ Res.\ Not.\ 2008, Art.\ ID rnn141, 10 pp.,
doi:10.1093/imrn/rnn141.
\bibitem{Ber93}
W.\ Bergweiler,
Iteration of meromorphic functions.
{\rm Bull.\ Amer.\ Math.\ Soc.\ (N.\ S.)}
29 (1993), 151--188.
\bibitem{BRS}
W.\ Bergweiler, P.\ J.\  Rippon and G.\ M.\  Stallard, Dynamics of
meromorphic functions with direct or logarithmic singularities. {\rm
Proc.\ London Math.\ Soc.} 97 (2008), 368--400.
\bibitem{Dom98}
P.\ Dom\'inguez,
Dynamics of transcendental meromorphic functions.
Ann.\ Acad.\ Sci.\ Fenn.\ Math.\ 23 (1998), 225--250.
\bibitem{Ere89}
A.\ E.\ Eremenko, On the iteration of entire functions, in {\rm
``Dynamical systems and ergodic theory''}. Banach Center Publications
23, Polish Scientific Publishers, Warsaw 1989, pp.\ 339--345.
\bibitem{EL}
A.\ E.\ Eremenko and M.\ Yu.\ Lyubich, Dynamical properties of some
classes of entire functions. {\rm Ann.\ Inst.\ Fourier} 42 (1992),
989--1020.
\bibitem{GO}
A.\ A.\ Goldberg and I.\ V.\ Ostrovskii,
Value distribution of meromorphic functions.
Transl.\ Math.\ Monographs 236, American Math.\ Soc., Providence, R.~I., 2008.
\bibitem{Hay64}
W.\ K.\ Hayman,
{\rm Meromorphic functions}.
Clarendon Press, Oxford, 1964.
\bibitem{Kotus95}
J.\ Kotus,
On the Hausdorff dimension of Julia sets of meromorphic functions, II.
Bull.\ Soc.\ Math.\ France  123  (1995),  33--46.
\bibitem{KU}
J.\ Kotus and M.\ Urba\'nski,
Hausdorff dimension and Hausdorff measures of Julia sets of elliptic
functions.
Bull.\ London Math.\ Soc.\ 35 (2003), 269--275.
\bibitem{KU2}
J.\ Kotus and M.\ Urba\'nski,
Fractal measures and ergodic theory of transcendental meromorphic functions,
in ``Transcendental Dynamics and Complex Analysis''.
London Math.\ Soc.\ Lect.\ Note Ser.\ 348.
Edited by P.\ J.\ Rippon
and G.\ M.\ Stallard,
Cambridge Univ.\ Press, Cambridge, 2008, pp.~251--316.
\bibitem{MaU}
D.\ Mauldin and M.\ Urba\'nski,
Dimensions and measures in infinite iterated function systems.
Proc.\ London Math.\ Soc.\ (3) 73 (1996), 105--154.
\bibitem{Ma}
V.\ Mayer,
The size of the Julia set of meromorphic functions.
Preprint, arXiv:math/0701256v1.
\bibitem{McM}
C.\ McMullen, Area and Hausdorff dimension of Julia sets of entire
functions. Trans.\ Amer.\ Math.\ Soc.\ 300 (1987), 329--342.
\bibitem{Nev}
R.\ Nevanlinna,
Eindeutige analytische Funktionen.
Springer, Berlin, Heidelberg,  1953.
\bibitem{RempeVanStrien}
L.\ Rempe and  S.\ van Strien,
Absence of line fields and Ma\~{n}\'e's theorem for non-recurrent 
transcendental functions.
Preprint,	arXiv:0802.0666v2.
\bibitem{RS}
P.\ J.\ Rippon and  G.\ M.\ Stallard, Iteration of a class  of
hyperbolic meromorphic  functions. {\rm  Proc.\ Amer.\ Math.\ Soc.}
127 (1999), 3251--3258.
\bibitem{Sch07}
H.\ Schubert, \"Uber die Hausdorff-Dimension der Juliamenge von
Funktionen endlicher Ordnung. Dissertation, University of Kiel,
2007.
\bibitem{Stallard90}
 G.\ M.\ Stallard,
Entire functions with Julia sets of zero measure.
Math.\ Proc.\ Cambridge Philos.\ Soc.\ 108 (1990), 551--557.
\bibitem{Stall}
 G.\ M.\ Stallard,
Dimensions of Julia sets of transcendental meromorphic functions,
in ``Transcendental Dynamics and Complex Analysis''.
London Math.\ Soc.\ Lect.\ Note Ser.\ 348.
Edited by P.\ J.\ Rippon
and G.\ M.\ Stallard,
Cambridge Univ.\ Press, Cambridge, 2008, pp.~425--446.
\bibitem{Tei}
O.\ Teichm\"uller,
Eine Umkehrung des zweiten Hauptsatzes der Wertverteilungstheorie.
Deutsche Math.\ 2 (1937), 96--107;
Gesammelte Abhandlungen,
Springer, Berlin, Heidelberg, New York, 1982, pp. 158--169.


\end{thebibliography}
\end{document}